\documentclass[12pt,a4paper,reqno]{amsart}
\usepackage{calc}
\usepackage{cite}
\usepackage{amssymb,amsmath,latexsym,amsthm,enumerate,amsfonts,graphicx,times}
\usepackage[mathscr]{euscript}
\usepackage{amssymb}
\usepackage{amsmath}

\def\dj{d\kern-0.4em\char"16\kern-0.1em}
\def\Dj{\mbox{\raise0.3ex\hbox{-}\kern-0.4em D}}

\def\be{\begin{equation}}
\def\ee{\end{equation}}
\def\bena{\begin{eqnarray*}}
\def\ena{\end{eqnarray*}}

\def\t{\tau}
\def\s{\sigma}

\def\suml{\sum\limits}

\def\dss{\displaystyle}

\newcommand{\WF}{\operatorname{WF}}

 \def\D{\mathcal{D}}
 \def\E{\mathcal{E}}
 \def\Rd{\mathbf{R}^d}
 \def\Z{\mathbf{Z}_+}

\def\N{\mathbf{N}}

\def\lf {\lfloor}
\def\rf{\rfloor}

\numberwithin{equation}{section}

\newtheorem{te}{Theorem}[section]
\newtheorem{lema}{Lemma}[section]
\newtheorem{prop}{Proposition}[section]
\newtheorem{cor}{Corollary}[section]

\theoremstyle{definition}

\newtheorem{de}{Definition}[section]
\theoremstyle{remark}
\newtheorem{rem}{Remark}[section]

\title{\textbf{A Paley-Wiener theorem in extended Gevrey regularity}}

\author{Stevan Pilipovi\' c}

\address{Department of Mathematics and Informatics,
University of Novi Sad, Novi Sad, Serbia}

\email{stevan.pilipovic@dmi.uns.ac.rs}

\author{Nenad Teofanov}

\address{Department of Mathematics and Informatics,
University of Novi Sad, Novi Sad, Serbia}

\email{nenad.teofanov@dmi.uns.ac.rs}

\author{Filip Tomi\'c}

\address{Faculty of Technical Sciences,
University of Novi Sad, Novi Sad, Serbia}

\email{filip.tomic@uns.ac.rs}



\begin{document}
\begin{abstract}
In this paper we introduce appropriate associated function to the sequence $M_p=p^{\t p^{\s}}$, $p\in \N$, $\t>0$, $\s>1$,
and derive  its sharp asymptotic estimates in terms of the Lambert $W$ function.
These estimates are used to prove a Paley-Wiener type theorem for  compactly supported
functions from extended Gevrey classes.
As an application, we discuss properties of the corresponding wave front sets.
\end{abstract}

\maketitle

\par

\section{Introduction}

{\em Paley-Wiener type theorems} describe relationship between the rate of decay at infinity for functions (distributions)
and regularity properties of their Fourier transforms.  In its simplest form these type of theorems state the following: if $\varphi$ is a smooth compactly supported function, then its Fourier transform $\widehat \varphi (\xi)$ decreases at infinity faster than $\dss (1+|\xi|)^{-N}$ for any $N\in \N$. These results give important theoretical insight to the objects under the study, and
could be applied for example in qualitative analysis of (hyperbolic) partial differential equations
(cf. \cite[Chapter 7.3]{HermanderKnjiga}, \cite[Chapter 7.2]{Str}), and structure theorems (cf. \cite{Komatsuultra1}).
Associated function to a given positive and increasing sequence $M_p$, $p\in\N$, plays an essential roll for the proofs of Paley-Wiener type theorems in the context of ultradifferentiable classes of functions (\cite{GelfandShilov, Komatsuultra1, Rodino}).
For example, if $\varphi$ is a compactly supported Gevrey function of Roumieu (resp. Beurling) type with index $\t>1$ (cf. \cite{KomatsuNotes, Rodino}) then
\be
\label{prvaocena}
|\widehat\varphi (\xi)|\leq A e^{-B |\xi|^{1/\t}},\quad \xi\in \Rd,
\ee
for some constants $A,B>0$ (resp. for every $B>0$ there exists $A>0$). The function $B |\xi|^{1/\t}$ (cf. \cite{GelfandShilov, KomatsuNotes}) in the exponent of \eqref{prvaocena} precisely describes the asymptotic behavior of the associated function to the Gevrey sequence $M_p=p^{\t p}$, $p\in \N$.

Classes of \emph{extended Gevrey functions} and their defining sequences  $M^{\t,\s}_p=p^{\t p^{\s}}$, $p\in \N$, $\t>0$, $\s>1$, are introduced and investigated in \cite{PTT-01, PTT-02, PTT-03, TT0, TT, FT}. Such classes are of interest for the analysis of certain
strictly hyperbolic equations. We refer to \cite{CL}, where the the sequence $M^{1,2}_p=p^{p^2}$
(and the corresponding associated function) gives rise to the well posedness of the related Cauchy problem.

The aim of this paper is to extend the notion of the associated function to sequences $M^{\t,\s}_p =p^{\t p^{\s}}$, $\t>0$, $\s>1$, and give an estimate which can be used in the proof of Paley-Wiener type results for the extended Gevrey classes.

The paper is organized as follows. We end this section with some notation, and the definition and
basic facts on the Lambert $W$ function (cf. \cite{LambF}) which is the main technical tool in our investigations.
Section \ref{sec01} contains  the basic facts concerning the spaces of ultradifferentiable functions defined by the means of sequences
$M^{\t,\s}_p =p^{\t p^{\s}}$, $\t>0$, $\s>1$. These classes of smooth functions
contain Gevrey classes, and for that reason we refer to them as to extended Gevrey classes, see Remark \ref{Gevrey-ext-Gevrey}.
We introduce appropriate associated function $T_{\t,\s,h}(k)$ for the sequence $p^{\t p^{\s}}$, $p\in \N$ in Subsection \ref{subsectionT},
and,  as the first  main result of the paper, we derive the essential (sharp) estimates for its asymptotic behavior,  Theorem \ref{propozicija}.
In Section \ref{PaleyWiener} we  use the asymptotic estimates from Theorem \ref{propozicija} to
prove a Paley-Wiener type theorem for $\D_{\t,\s}(U)$ (see Theorem \ref{PaleyWiener} and its Corollary \ref{Cor:PaleyWiener}), which is our second main result.
As an application, in Section \ref{section:WFS} we introduce wave front sets related to the extended Gevrey classes and prove Theorem \ref{NezavisnostWFThm} which states that the definition of such wave front sets is independent on the choice of the cutoff function  $\phi\in \D_{\t,\s}^K $. We present the proof of
 Theorem \ref{propozicija} in Section \ref{Sec:Proof}.

\subsection{Preliminaries}\label{Notacija}
We denote by ${\bf N}$, $\Z$, ${\bf R}$, ${\bf C}$ the sets of nonnegative
integers, positive integers, real numbers and complex numbers, respectively. For $x \in \Rd$
we put  $\langle x \rangle=(1+|x|^2)^{1/2}$.
The floor and the ceiling functions, the integer parts of  $x\in {\bf R}_+$,
are denoted by $\lf x \rf:=\max\{m\in \N\,:\,m\leq x\}$ and $\lceil x \rceil:= \min\{m\in \N\,:\,m\geq x\}$, respectively. For a multi-index
$\alpha=(\alpha_1,\dots,\alpha_d)\in {\bf N}^d$ we write
$\partial^{\alpha}=\partial^{\alpha_1}\dots\partial^{\alpha_d}$, $\dss D^{\alpha}= (-i )^{|\alpha|}\partial^\alpha$, and
$|\alpha|=|\alpha_1|+\dots |\alpha_d|$.
We put  $ A\lesssim B $ when $ A \leq C \cdot B $ for some positive constant $C$.
We write $ A \asymp B $ if $ A\lesssim B $  and  $ B\lesssim A $.

The Fourier-Laplace transform is denoted by
$$\dss \widehat u(\eta)=\int_{\Rd} u(x)e^{ i x\cdot\eta}\,dx  \;\;\; \eta \in {\mathbf C}^d,\quad u\in L^1 (\Rd). $$
For locally convex topological spaces   $X$ and $Y$ we write $X\hookrightarrow Y$
when $X\subseteq Y$ and the identity mapping from $X$ to $Y$ is continuous.

\par

The Lambert $W$ function is defined as the inverse function of $z e^{z}$, $z\in {\bf C}$, wherefrom the following property holds:
\be
\label{osobinaLambert}
\dss x=W(x)e^{W(x)}, \quad x\geq 0.
\ee
We denote its principal (real) branch by $W(x)$, $x\geq 0$ (see \cite{LambF}).
It is a continuous,  increasing and concave function on $[0,\infty)$, $W(0)=0$, $W(e)=1$, and $W(x)>0$, $x>0$.

It can be shown that $W$ can be represented in the form of the absolutely convergent series
$$
W(x)=\ln x-\ln (\ln x)+\sum_{k=0}^{\infty}\sum_{m=1}^{\infty}c_{km}\frac{(\ln(\ln x))^m}{(\ln x)^{k+m}},\quad x\geq x_0>e,
$$
with suitable constants $c_{km}$ and  $x_0 $, wherefrom  the following  estimates hold:

\be
\label{sharpestimateLambert}
\ln x -\ln(\ln x)\leq W(x)\leq \ln x-\frac{1}{2}\ln (\ln x), \quad x\geq e.
\ee
The equality in \eqref{sharpestimateLambert} holds if and only if $x=e$.
We refer to \cite{HoHa, LambF} for more details about the Lambert $W$ function.

\section{Extended Gevrey classes ${\E}_{\t, \s}(U)$ and ${\D}_{\t, \s}(U)$}
\label{sec01}

We employ Komatsu's approach \cite{Komatsuultra1} to spaces of
ultradifferentiable functions and recall the definition of test function spaces denoted by  ${\E}_{\t, \s}(U)$
and ${\D}_{\t, \s}(U)$
via defining sequences of the form
$M_p^{\t,\s}=p^{\t p^{\s}}$, $p\in \N $, depending on parameters $\t>0$ and  $\s>1$, \cite{PTT-02}.
The flexibility obtained by introducing the two-parameter dependence
enables the study of smooth functions which are less regular than the Gevrey functions.
When $\t>1$ and  $\s=1$, we  recapture the usual  Gevrey classes.

First we recall the essential properties of the defining sequences $M_p^{\t,\s}$.
We refer to \cite{PTT-01} for the proof of the next Lemma.

\begin{lema}
\label{osobineM_p_s}
Let $\tau>0$, $\s>1$ and $M_p^{\tau,\s}=p^{\tau p^{\s}}$, $p\in \Z$, $M_0^{\tau,\s}=1$.
Then there exists an  increasing sequence of positive numbers $C_q$, $q\in \N$, and a constant $C>0$ such that:
\vspace{2mm}\\
\vspace{1mm}
$(M.1)$ $(M_p^{\t,\s})^2\leq M_{p-1}^{\t,\s}M_{p+1}^{\t,\s}$, $p\in \Z$\\
\vspace{1mm}
$\overline{(M.2)}$ $M_{p+q}^{\t,\s}\leq C^{p^{\s}+q^{\s}}M_p^{\t 2^{\s-1},\s}M_q^{\t 2^{\s-1},\s}$, $p,q\in \N$,\\
\vspace{1mm}
$\overline{(M.2)'}$ $M_{p+q}^{\t,\s}\leq C_q^{p^{\s}}M_p^{\t,\s}$, $p,q\in \N$,\\
\vspace{1mm}
$(M.3)'$
$ \displaystyle
\suml_{p=1}^{\infty}\frac{M_{p-1}^{\t,\s}}{M_p^{\t,\s}} <\infty.
$ Moreover, $\dss \frac{M_{p-1}^{\t,\s}}{M_p^{\t,\s}}\leq \frac{1}{(2p)^{\tau (p-1)^{\s-1}}}$, $p\in \N$.
\end{lema}

Let $\tau,h>0$, $\s>1$ and let $K\subset \subset \Rd$ be a regular compact set.  By ${\E}_{\t, {\s},h}(K)$
we denote
the Banach space of  functions $\phi \in  C^{\infty}(K)$ such that
\begin{equation} \label{Norma}
\| \phi \|_{{\E}_{\t, {\s},h}(K)}=\sup_{\alpha \in \N^d}\sup_{x\in K}
\frac{|\partial^{\alpha} \phi (x)|}{h^{|\alpha|^{\s}}  M_{|\alpha|} ^{\t,\s} }<\infty.\,
\end{equation}
Then
$$ \displaystyle
{\E}_{\t_1, {\s_1},h_1}(K)\hookrightarrow {\E}_{\t_2,
{\s_2},h_2}(K), \;\;\;
0<h_1\leq h_2, \; 0<\t_1\leq\t_2, \; 1<\s_1\leq \s_2,
$$
where $\hookrightarrow$ denotes the strict and dense inclusion.

The set of functions $ \phi \in {\E}_{\t,\s,h}(K)$ whose support is contained in $K$ is denoted by  ${\D}^K_{\t, \s,h}$.

Let $U$ be an open set $\Rd$ and $ K \subset \subset U$. We define families of spaces
by introducing the following projective and inductive limit topologies:
\begin{equation*}
\label{NewClassesInd} {\E}_{\{\t,
\s\}}(U)=\varprojlim_{K\subset\subset U}\varinjlim_{h\to
\infty}{\E}_{\t, {\s},h}(K),
\end{equation*}
\begin{equation*}
\label{NewClassesProj} {\E}_{(\t,
\s)}(U)=\varprojlim_{K\subset\subset U}\varprojlim_{h\to 0}{\E}_{\t,
{\s},h}(K),
\end{equation*}
\begin{equation*}
\label{NewClassesInd2} {\D}_{\{\t,
\s\}}(U)=\varinjlim_{K\subset\subset U} {\D}^K_{\{\t, \s\}}
=\varinjlim_{K\subset\subset U} (\varinjlim_{h\to\infty}{\D}^K_{\t,
\s,h})\,,
\end{equation*}
\begin{equation*}
\label{NewClassesProj2} {\D}_{(\t,
\s)}(U)=\varinjlim_{K\subset\subset U} {\D}^K_{(\t, \s)}
=\varinjlim_{K\subset\subset U} (\varprojlim_{h\to 0}{\D}^K_{\t,
\s,h}).
\end{equation*}
We will use abbreviated notation $ \t,\s $ for
$\{\t,\s\}$ or $(\t,\s)$.
The spaces ${\E}_{\t, \s}(U)$, ${\D}^K_{\t, \s}$ and ${\D}_{\t, \s}(U)$
are nuclear, cf. \cite{PTT-01}.
We refer to \cite{PTT-01, PTT-02, PTT-03, TT0, TT, FT} for other properties of those spaces.

\par

\begin{rem} \label{Gevrey-ext-Gevrey}
If $ \t > 1 $ and $\s = 1$,
then $ {\E}_{\{\t, 1\}}(U)={\E}_{\{\t\}}(U)$  is the Gevrey class,
and $\D_{\{\t,1\}}(U)=\D_{\{\t\}}(U)$
is its subspace of compactly supported functions in $\E_{\{\t\}}(U)$. If $0<\t\leq 1$, then
$ {\E}_{\t, 1}(U)$ consists of quasianalytic functions. In particular, $\dss
\D_{\t,1}(U)=\{0\}$ when $0<\t\leq 1$, and ${\E}_{\{1, 1\}}(U)= {\E}_{\{1\}}(U)$ is the space of analytic functions on $U$.
\end{rem}

The space $ \E_{\{1,2\}}(U)$ appears in \cite{CL} in the study of  strictly hyperbolic equations.
For this class the assumptions that coefficients belong to $ \E_{\{1,2\}}(U)$ with respect to space variable, and with the certain lower regularity in time, imply that the corresponding Cauchy problem is well posed in appropriate solution spaces.

\par

In the following proposition, main embedding properties between the above introduced families
are captured.

\begin{prop}
\label{detectposition} \cite{PTT-02} Let $\s_1\geq 1$. Then for every $\s_2>\s_1$
and $\t>0$
\begin{equation*} \label{Theta_S_embedd}
\varinjlim_{\t\to \infty}{\E}_{\t,
{\s_1}}(U)\hookrightarrow \varprojlim_{\t\to 0^+} {\E}_{\t,
{\s_2}}(U).
\end{equation*}
Moreover, if $0<\t_1<\t_2$, then
\be \label{RoumieuBeurling} \E_{\{\t_1,\s\}}(U)\hookrightarrow
\E_{(\t_2,\s)}(U)\hookrightarrow \E_{\{\t_2,\s\}}(U), \;\;\; \s\geq 1,\nonumber \ee
and
$$
\varinjlim_{\t\to \infty}{\E}_{\{\t, {\s}\}}(U)= \varinjlim_{\t\to \infty} {\E}_{(\t, {\s})}(U),
$$
$$
\varprojlim_{\t\to 0^+}{\E}_{\{\t, {\s}\}}(U)= \varprojlim_{\t\to 0^+} {\E}_{(\t, {\s})}(U), \;\;\; \s\geq 1.
$$
\end{prop}

We conclude that
\begin{equation*}
\label{Theta_S_embedd-2}
{\E}_{\tau_0, {\s_1}}(U)\hookrightarrow \bigcap_{\tau> \tau_0} {\E}_{\tau, {\s_1}}(U)
\hookrightarrow {\E}_{\tau_0, {\s_2}}(U),
\end{equation*}
for any $\tau_0>0$ whenever $\s_2>\s_1\geq 1$. In particular (see Remark \ref{Gevrey-ext-Gevrey}),
\begin{equation*}
\label{GevreyNewclass}
\varinjlim_{t\to\infty} \E_{\{t\}}(U)\hookrightarrow {\E}_{\tau, \s}(U)
\hookrightarrow  C^{\infty}(U), \;\;\;
\tau>0, \; \s>1,
\end{equation*}
so that  the regularity in ${\E}_{\tau, \s}(U)$ can be thought of as an extended Gevrey regularity.

\par

The non-quasianalyticity condition $(M.3)'$ provides the existence of
partitions of unity in  $\E_{\{\t,\s\}}(U)$ which we formulate in the next Lemma.

\begin{lema} [\cite{PTT-01}]
\label{teoremaKompaktannosac}
Let $\t>0$ and $\s>1$. Then there exists a compactly supported function $\phi\in {\E_{\{\t,\s\}}}(U)$ such that $0\leq\phi\leq 1$ and $\int_{\Rd}\phi\,dx=1$.
\end{lema}

Of course, any compactly supported Gevrey function from $\E_{\{\t\}}(U)$ belongs to $ {\D}_{\{\t, \s\}}(U)$ as well.
However, in the proof of Lemma \ref{teoremaKompaktannosac} given in \cite{PTT-01} we
constructed a compactly supported function in ${\D}_{\{\t, \s\}}(U)$  which
does not belong to   ${\D}_{\{t\}}(U)$, for any $t>1$.

Note that the additional exponent $\s$, which appears in the power of term $h$ in \eqref{Norma},
makes the definition of $\E_{\{\t,\s\}}(U)$ different from the definition of  Carleman class, cf. \cite{HermanderKnjiga}.
This difference is essential for many calculations. For example,  defining sequences for
Carleman classes satisfy Komatsu's condition (M.2)' known as
``stability under differential operators``, while  $M_p^{\tau,\s}$ do not satisfy  (M.2)'
for $\tau>0$ and $\s>1$. However, we have the following ``stability properties``.

If $\dss P=\sum_{|\alpha|\leq m}a_{\alpha}(x)\partial^{\alpha}$ is a partial differential operator of order $m$ with $a_{\alpha}\in \E_{\t,\s}(U)$, then  $P\,:\,\E_{\t,\s}(U)\to \E_{\t,\s}(U)$ is a continuous linear map with respect to the topology of $\E_{\t,\s}(U)$.
In particular, ${\E}_{\t, \s}(U)$ is closed under pointwise
multiplications and finite order differentiation, see \cite[Theorem 2.1]{TT}.
For operators of ``infinite order`` continuity properties are slightly different, see \cite{PTT-02}.

\subsection{Extended associated functions} \label{subsectionT}

Let $M_p$, $p\in \N$, be the sequence of positive numbers such that $M^{1/p}_p$ is bounded from below and $M_0 = 1$. The \emph{associated function} to $M_p$ is defined by
$$T(k)=\sup_{p\in\N}\ln \frac{k^{p}}{M_p},\quad k>0.$$
These functions play an essential role in theory  of ultradistributions, see \cite{Komatsuultra1}.
A convenient modification of the above definition is given as follows.

\begin{de}
Let $\tau>0$, $\s>1$ and $M_p^{\tau,\s}=p^{\tau p^{\s}}$, $p\in \Z$, $M_0^{\tau,\s}=1$.
The extended associated function related to the sequence $M_p^{\t,\s}$, is given by
$$
\dss T_{\t,\s,h}(k)=\sup_{p\in \N}\ln_+\frac{h^{p^{\s}}k^{p}}{M_p^{\t,\s}}, \;\;\; h,k>0,
$$
where $\dss \ln_+ A=\max\{0, \ln A\}$, for $A>0$.
\end{de}

Note that for fixed $\t,h>0$ and $\s>1$,  $T_{\t,\s,h}(k)$ is always positive for $k$ sufficiently large.

In the following lemma we compare the usual associated functions to the sequences $p^{\t p^{\s}}$, $\t>0$, $\s>1$, and
Gevrey sequences $p^{t p}$, $t>1$ for which the associated function satisfies
$T_t (k) \asymp C k^{1/t}$, when $k\to \infty$, \cite{Rodino}.

\begin{lema}
\label{Lema:comparison}
Let $\t>0$, $\s>1$ and $M_p^{\t,\s}=p^{\t p^{\s}}$, $p\in \N$. Then for any $t>1$ there exists constant $C>0$ such that
$$
\sup_{p\in \N}\ln\frac{k^{p}}{M_p^{\t,\s}} < C  k^{1/t}, \quad k>0.
$$
\end{lema}

\begin{proof}
Note the for arbitrary $t>0$ there exists constant $H>0$ such that
$$p^{t p}< H p^{\t p^{\s}},\quad \t>0,\, \s>1,\, p\in \N,$$
wherefrom
$$
\ln \frac{k^p}{p^{\t p^{\s}}}< \ln \frac{k^p}{p^{t p}}+\ln H,\quad k>0.
$$
Since
$$T_t (k)=\dss\sup_{p\in \N}\ln \frac{k^p}{p^{t p}} \asymp C k^{1/t},\quad k\to \infty,$$ for some $C>0$  the proof is immediate.
\end{proof}

By Lemma \ref{Lema:comparison} it follows that the growth of $\dss T_{\t,\s,h}(k)$, $k>0$, is slower then $k^{1/t}$ for any $t>0$. The precise result is given in the following Theorem.

For given $\s>0$, $\t,h>0$ let
$$\dss {\mathfrak R}(h,\cdot):=h^{-\frac{\s-1}{\t}}e^{\frac{\s-1}{\s}}\frac{\s-1}{\t \s}\ln k,\quad k>e. $$

\begin{te}
\label{propozicija}
Let $h>0$, $\t>0$ and $\s>1$. Then
\begin{multline}
\label{nejednakostzaTeoremu1}
\tilde{A}_{\t,\s,h}\exp\Big\{( 2^{\s-1 }\t)^{-\frac{1}{\s-1}}{\Big(\frac{\s-1}{ \s}\Big)^{\frac{\s}{\s-1}}\, {W^{-\frac{1}{\s-1}}({{\mathfrak R}(h,k)})}\,{\ln}^{\frac{\s}{\s-1}}k }\Big\}\leq
e^{T_{\t,\s,h}(k)}\\ \leq
 A_{\t,\s,h}\exp\Big\{{\Big(\frac{\s-1}{\t \s}\Big)^{\frac{1}{\s-1}}\, {W^{-\frac{1}{\s-1}}({{\mathfrak R}(h,k)})}\,{\ln}^{\frac{\s}{\s-1}}k }\Big\}, \quad k>e,
\end{multline} for some $A_{\t,\s,h}, \tilde{A}_{\t,\s,h}>0$.
\end{te}

The proof is given in Section \ref{Sec:Proof}.

Notice that the right-hand side of \eqref{nejednakostzaTeoremu1} states that for any given $h>0$ there exists $A_h>0$ such that

\be
\label{ekvivalentnaNejedn}
\sup_{k>0} k^p \exp\Big\{-{\Big(\frac{\s-1}{\t \s}\Big)^{\frac{1}{\s-1}}\, {W^{-\frac{1}{\s-1}}({\mathfrak R}(h,k))}\,{\ln}^{\frac{\s}{\s-1}}k }\Big\}\leq A_h (1/h)^{p^{\s}} p^{\t p^{\s}},
\ee for $p\in \N$.

We finish this section with some useful remarks.

\begin{rem}
When $\t>0\,,h>0$ and $1<\s<2$, the first inequality in \eqref{nejednakostzaTeoremu1} can be improved. In particular, since
the second term in \eqref{Taylor} is equal to zero when $1<\s<2$, it can be shown that
\begin{multline*}
\tilde{A}_{\t,\s,h}\exp\Big\{{\Big(\frac{\s-1}{\t \s}\Big)^{\frac{1}{\s-1}}\, {W^{-\frac{1}{\s-1}}({{\mathfrak R}(h,k)})}\,{\ln}^{\frac{\s}{\s-1}}k }\Big\}\leq
e^{T_{\t,\s,h}(k)}\\ \leq
 A_{\t,\s,h}\exp\Big\{{\Big(\frac{\s-1}{\t \s}\Big)^{\frac{1}{\s-1}}\, {W^{-\frac{1}{\s-1}}({{\mathfrak R}(h,k)})}\,{\ln}^{\frac{\s}{\s-1}}k }\Big\}, \quad k>e,
\end{multline*}
with suitable $A_{\t,\s,h}> \tilde{A}_{\t,\s,h}>0$.
\end{rem}

\begin{rem}
In the view of \eqref{sharpestimateLambert}, for every  $h>0$ we have
\begin{multline}
\label{asimptotskaocena}
 {W^{-\frac{1}{\s-1}}({\mathfrak R}(h,k))}\,{\ln}^{\frac{\s}{\s-1}}k
\asymp
\Big(  \frac{\ln k}{\ln (C_h \ln k)}\Big)^{\frac{1}{\s-1}}\ln k,\quad k\to \infty,
\end{multline} where $\dss C_h:=h^{-\frac{\s-1}{\t}}e^{\frac{\s-1}{\s}}\frac{\s-1}{\t \s}$.

Since $\lim_{k\to\infty} ( \ln k)^{1/(\s-1)} (\ln (C_h \ln k) )^{-1/(\s-1)} =\infty,$
for every $h>0$,
\eqref{asimptotskaocena} implies that for every $M>0$ there exists $B>0$ (depending on $h$ and $M$) such that
\be
\label{ocenaIntegral}
{W^{-\frac{1}{\s-1}}({\mathfrak R}(h,k))}\,{\ln}^{\frac{\s}{\s-1}}k> M\ln k,\quad k> B.
\ee
\end{rem}

\section{Paley-Wiener theorems}\label{PaleyWiener}

Next  we prove the Paley-Wiener theorem for the spaces $\D_{(\t,\s)}(U)$ and $\D_{\{\t,\s\}}(U)$.
We use some ideas presented in \cite{KomatsuNotes}.

\begin{te}
\label{TeoremaPaley}
Let $\t>0$, $\s>1$, $U$ be open set in $\Rd$ and $K\subset\subset U$. If $\varphi\in \D_{\t,\s,h}^{K}$ for some $h>0$,
then its Fourier-Laplace transform is an entire function and satisfies
\begin{multline}
\label{ocenaTeorema0}
|\widehat\varphi(\eta)|\leq \\ A_{\t,\s,h} \exp\Big\{-(\t 2^{\s-1 })^{-\frac{1}{\s-1}}\Big(\frac{\s-1}{\s}\Big)^{\frac{\s}{\s-1}}{W^{-\frac{1}{\s-1}}\Big({\mathfrak R}\Big(\frac{1}{2e h\sqrt{d}},e+|\eta|\Big)\Big)}{{\ln}}^{\frac{\s}{\s-1}}(e+|\eta|)+ H_K(\eta) \Big\}\\ \eta \in {\mathbf C}^d,
\end{multline}
for some $A_{\t,\s,h}>0$, where $\dss H_K(\eta)=\sup_{y\in K} {\rm Im}(y\cdot\eta)$.

Conversely, if an entire function $\dss{F}$ satisfies
\begin{multline}
\label{ocenaTeorema}
|F(\eta)|\leq A_{\t,\s,h} \exp\Big\{-\Big(\frac{\s-1}{\t \s}\Big)^{\frac{1}{\s-1}}{W^{-\frac{1}{\s-1}}\Big({\mathfrak R}\Big(\frac{2^{\t}}{h},e+|\eta|\Big)\Big)}{{\ln}}^{\frac{\s}{\s-1}}(e+|\eta|)+ H_K(\eta) \Big\},\\ \eta \in {\mathbf C}^d,
\end{multline}
for some $h>0$ and $A_{\t,\s,h}>0$, then  $F$ is the Fourier-Laplace transform of some  $\varphi\in \D_{2^{\s-1}\t,\s,h}^{K}$.
\end{te}

\begin{proof}
Let $K\subset\subset U$, $\t>0$ and $\s>1$ be fixed and let $\varphi\in \D_{\t,\s,h}^{K}$ for some $h>0$. Throughout the proof we will use the following simple inequalities:
$$\dss |\eta|^{|\alpha|}\leq (e+|\eta|)^{|\alpha|}\leq (2e)^{|\alpha|} |\eta|^{|\alpha|},$$
$$(1/\sqrt{d})|\eta|^{|\alpha|}\leq |\eta^{\alpha}|\leq |\eta|^{|\alpha|},\quad \alpha\in \N^d, \eta\in{\mathbf C}^d.$$
This gives
$$|(e+|\eta|)^{|\alpha|}\widehat\varphi(\eta)|\leq (2e{\sqrt{d}})^{|\alpha|} |\widehat{D^{\alpha}\varphi}(\eta)|\leq \sup_{x\in K}|D^{\alpha}\varphi(x)|\cdot\int_{K}e^{{\rm Im}(x\cdot\eta)} dx,$$
wherefrom
\begin{multline}
\label{NejednakostZaWF}
|\widehat\varphi(\eta)|
\leq \inf_{\alpha\in \N^d}\Big( \frac{(2eh\sqrt{d})^{|\alpha|^{\s}}|\alpha|^{\t|\alpha|^{\s}}}{(e+|\eta|)^{|\alpha|}}\Big) e^{H_K(\eta)},
\alpha\in \N^d,\,\, \eta\in {\mathbf C}^d.
\end{multline}
Now \eqref{ocenaTeorema0} follows directly from the left-hand side of \eqref{nejednakostzaTeoremu1}, since
$$ \inf_{\alpha\in \N^d} \frac{(2eh\sqrt{d})^{|\alpha|^{\s}}|\alpha|^{\t|\alpha|^{\s}}}{(e+|\eta|^{|\alpha|})}=\Big(\sup_{\alpha\in \N^d} \frac{(e+|\eta|)^{|\alpha|}}{(2eh\sqrt{d})^{|\alpha|^{\s}}|\alpha|^{\t|\alpha|^{\s}}}\Big)^{-1}.$$

To prove the second part of the theorem, set $\eta=\xi +i \mu$, for $\xi, \mu \in \Rd$ and note that $\dss H_K(\eta)=\sup_{y\in K} {\rm Im}(y\cdot\eta)=\sup_{y\in K} y\cdot \mu$.
Let
\be
\label{definicijaFunkcijeFi}
\dss \varphi (x):=(2\pi)^{-d}\int_{\Rd}F (\xi) e^{i x\xi}d\xi, \quad x\in \Rd,
\ee and note that for some $h>0$ \eqref{ocenaTeorema} implies
\begin{multline}
\label{ocenePaley}
|D^{\alpha} \varphi(x)|=\Big| \int_{\Rd} \xi^{\alpha} F (\xi)e^{i x\xi}d\xi \Big|\\
\leq \int_{\xi\in \Rd}|\xi^{\alpha}|\exp\Big\{-\Big(\frac{\s-1}{\t \s}\Big)^{\frac{1}{\s-1}}{W^{-\frac{1}{\s-1}}\Big({\mathfrak R}\Big(\frac{2^{\t}}{h},e+|\xi|\Big)\Big)}{{\ln}}^{\frac{\s}{\s-1}}(e+|\xi|)\Big\}d\xi\\
\leq \sup_{\xi\in \Rd}(e+|\xi|)^{|\alpha|}\exp\Big\{-\Big(\frac{\s-1}{2^{\s-1}\t \s}\Big)^{\frac{1}{\s-1}}{W^{-\frac{1}{\s-1}}\Big({\mathfrak R}\Big(\frac{2^{\t}}{h},e+|\xi|\Big)\Big)}{{\ln}}^{\frac{\s}{\s-1}}(e+|\xi|) \Big\}\\
\times \int_{\Rd} \exp\Big\{-\Big(\frac{\s-1}{2^{\s-1}\t \s}\Big)^{\frac{1}{\s-1}}{W^{-\frac{1}{\s-1}}\Big({\mathfrak R}\Big(\frac{2^{\t}}{h},e+|\xi|\Big)\Big)}{{\ln}}^{\frac{\s}{\s-1}}(e+|\xi|) \Big\}d\xi, \quad x\in \Rd.
\end{multline}
By choosing $M=d+1$ in \eqref{ocenaIntegral} we conclude that there is a constant $B>0$ such that
$$ \exp\Big\{-\Big(\frac{\s-1}{2^{\s-1}\t \s}\Big)^{\frac{1}{\s-1}}{W^{-\frac{1}{\s-1}}\Big({\mathfrak R}\Big(\frac{2^{\t}}{h},\langle \xi\rangle\Big)\Big)}{{\ln}}^{\frac{\s}{\s-1}}\langle\xi\rangle \Big\}\leq e^{-(d+1)\ln \langle\xi\rangle}=\langle\xi\rangle^{-d-1},$$
when $\quad |\xi|>B$, wherefrom
\be
\label{integralOcena}
\int_{\Rd} \exp\Big\{-\Big(\frac{\s-1}{2^{\s-1}\t \s}\Big)^{\frac{1}{\s-1}}{W^{-\frac{1}{\s-1}}\Big({\mathfrak R}\Big(\frac{2^{\t}}{h},e+|\xi|\Big)\Big)}{{\ln}}^{\frac{\s}{\s-1}} (e+|\xi|) \Big\}d\xi\leq C_h,
\ee
for some $C_h>0$. Moreover, since
$$\dss {\mathfrak R}\Big(\frac{2^{\t}}{h},e+|\xi|\Big)=h^{\frac{\s-1}{\t}}e^{\frac{\s-1}{\s}}\frac{\s-1}{2^{\s-1}\t \s}\ln (e+|\xi|), \quad h>0,\,\xi\in\Rd,$$
the inequality \eqref{ekvivalentnaNejedn} implies that for a given $h>0$ there exists $A_h>0$ such that

\begin{multline}
\label{supremumOcena}
 \sup_{\xi\in \Rd}|\xi|^{|\alpha|}\exp\Big\{-\Big(\frac{\s-1}{2^{\s-1}\t \s}\Big)^{\frac{1}{\s-1}}{W^{-\frac{1}{\s-1}}\Big({\mathfrak R}\Big(\frac{2^{\t}}{h},e+|\xi|\Big)\Big)}{{\ln}}^{\frac{\s}{\s-1}}(e+|\xi|) \Big\}\\
\leq A_h h^{|\alpha|^{\s}} |\alpha|^{2^{\s-1}\t |\alpha|^{\s}}, \quad \alpha\in \N^d.
\end{multline}

Now \eqref{ocenePaley}, \eqref{integralOcena} and \eqref{supremumOcena} imply
$$
\sup_{x\in \Rd} |D^{\alpha} \varphi(x)|\leq \widetilde{A_h} h^{|\alpha|^{\s}} |\alpha|^{2^{\s-1}\t |\alpha|^{\s}}, \quad\alpha\in \N^d,
$$
for some $\widetilde{A_h}>0$.

It remains to show that ${\rm supp}\varphi \subseteq K$. Arguing as in the proof of \cite[Theorem 7.3.1]{HermanderKnjiga} we may shift the integration in \eqref{definicijaFunkcijeFi} so that

\be
\label{ProsirenoPhi1}
\dss \varphi (x)=(2\pi)^{-d}\int_{\Rd}F (\xi+i\mu) e^{i x (\xi+i \mu)}d\xi, \quad x\in \Rd,
\ee for any $\mu\in \Rd$. This is possible since \eqref{ocenaTeorema} and \eqref{ocenaIntegral} implies
\be
\label{ProsirenoPhi2}
 |F(\eta)|\leq C  e^{H_K(\eta)} \langle \eta \rangle^{-d-1}\leq C e^{\sup_{y\in K} (y\cdot \mu)} \langle \xi \rangle^{-d-1},\quad \eta=\xi+i \mu,
\ee for some $C>0$, where the last inequality follows from  $|\eta|^2=|\xi|^{2}+|\mu|^2>|\xi|^{2}$.
Now by \eqref{ProsirenoPhi1} and \eqref{ProsirenoPhi2} we have
\be
\label{ProsirenoPhi3}
|\varphi(x)|\leq C  \exp\{{{\sup_{y\in K} (y\cdot \mu)}-x\cdot \mu}\} \int_{\Rd}\langle \xi \rangle^{-d-1}d\xi
\ee
$$
=\tilde{C} \exp\{{{\sup_{y\in K} (y\cdot \mu)}-x\cdot \mu}\}, \;\;\; x, \mu \in \Rd
$$
for a suitable constant $\tilde{C}>0$.

If $x\not\in K$, then \cite[Theorem 4.3.2]{HermanderKnjiga} implies that there exists $\mu_0\in \Rd$ such that $\dss x\cdot \mu_0>\sup_{y\in K} (y\cdot \mu_0)$ and choosing $\mu=t \mu_0$, $t>0$ in \eqref{ProsirenoPhi3}, we have
$$
|\varphi(x)|\leq \tilde{C} \exp\Big\{t\Big({{\sup_{y\in K} (y\cdot \mu_0)}-x\cdot \mu_0}\Big)\Big\}\to 0, \quad t\to \infty.
$$
Thus, ${\rm supp}\varphi \subseteq K$, and the proof is finished.
\end{proof}

\begin{cor} \label{Cor:PaleyWiener}
Let $\t>0$, $\s>1$, $U$ be open set in $\Rd$ and $K\subset\subset U$. If $\varphi\in \D_{\{\t,\s\}}(U)$
($\varphi\in \D_{(\t,\s)}(U)$ respectively)
then its Fourier-Laplace transform is an entire function and satisfies
\begin{multline}
\label{ocenaTeorema0}
|\widehat\varphi(\eta)|\leq \\ A_{\t,\s,h} \exp\Big\{-(\t 2^{\s-1 })^{-\frac{1}{\s-1}}\Big(\frac{\s-1}{\s}\Big)^{\frac{\s}{\s-1}}{W^{-\frac{1}{\s-1}}\Big({\mathfrak R}\Big(\frac{1}{2e h\sqrt{d}},e+|\eta|\Big)\Big)}{{\ln}}^{\frac{\s}{\s-1}}(e+|\eta|)+ H_K(\eta) \Big\}\\ \eta \in {\mathbf C}^d,
\end{multline}
for some $h>0$, (for every $h>0$, respectively) and for some constant $A_{\t,\s,h}>0$,
where $\dss H_K(\eta)=\sup_{y\in K} {\rm Im}(y\cdot\eta)$.

Conversely, if an entire function $\dss{F}$ satisfies
\begin{multline}
\label{ocenaTeorema}
|F(\eta)|\leq A_{\t,\s,h} \exp\Big\{-\Big(\frac{\s-1}{\t \s}\Big)^{\frac{1}{\s-1}}{W^{-\frac{1}{\s-1}}\Big({\mathfrak R}\Big(\frac{2^{\t}}{h},e+|\eta|\Big)\Big)}{{\ln}}^{\frac{\s}{\s-1}}(e+|\eta|)+ H_K(\eta) \Big\},\\ \eta \in {\mathbf C}^d,
\end{multline}
for some $h>0$ for every $h>0$, respectively) and for some constant $A_{\t,\s,h}>0$, then  $F$ is the Fourier-Laplace transform of  some $\varphi\in \D_{\{\t,\s\}}(U)$ ($\varphi\in \D_{(\t,\s)}(U)$ respectively).
\end{cor}

\section{Wave front sets} \label{section:WFS}

Let $\t>0$, $\s>1$, $\overline{\Omega} \subseteq K\subset \subset U\subseteq\Rd$, where $\Omega$ and $U$ are open sets in $\Rd$, and
let $u\in \D'(U)$. We studied in \cite{PTT-01} the nature of regularity related to the condition
 \begin{equation}
\label{uslov3'}
|\widehat u_N(\xi)|\leq A\, \frac{h^{N} N!^{\t/\s}}{|\xi|^{\lfloor N^{1/\s} \rfloor}}, \quad N\in { \N},\,|\xi|\geq B>0,
\end{equation} where $\{u_N\}_{N\in \N}$ is bounded sequence in $\E'(U)$ such that $u_N=u$ in $\Omega$ and $A,h$ are some positive constants.

The sequence $ N\in { \N}$  in the right hand side of (\ref{uslov3'}) can be replaced by another
positive and increasing sequence $a_N$ such that $a_N\to \infty$, $N\to \infty$, and which gives the same asymptotic behavior of $ |\widehat u_N(\xi)|$ when $N\to \infty$.
This change of variables called {\em enumeration}. The effect is  ``speeding up`` or ``slowing down``
the decay estimates of single members of the corresponding sequences,
without changing the asymptotic behavior of the whole sequence when $N \rightarrow \infty$.
After applying the enumeration $N\to a_N$, we can write  again $u_N$ instead of $u_{a_N}$, since we are only interested in the  asymptotic behavior.

For example, Stirling's formula  and enumeration $N\to N^{\s}$ applied to (\ref{uslov3'}) give an equivalent  estimate of the form
$$
|\widehat u_N(\xi)|\leq A_1\, \frac{h_1^{N^{\s}} N^{\t N^{\s}}}{|\xi|^{N}},
\quad N\in { \N},\,|\xi|\geq B>0,
$$
for some constants $A_1, h_1 > 0$. We refer to  \cite{PTT-02} for more details on enumeration.

Wave front sets  ${\WF}_{\{\t,\s\}}(u)$ are introduced in \cite{PTT-02}
in the study of local regularity in ${\E}_{\{\tau, \s\}}(U)$, see also \cite{PTT-03, TT}.
For the definition, appropriate sequences of cutoff functions are used in a similar way as it is done in
\cite{HermanderKnjiga} in the context of  analytic wave front set ${\WF}_A$.
We recall the definition of ${\WF}_{\{\t,\s\}}(u)$.

\begin{de}
\label{Wf_t_s1}
Let $u\in \D'(U)$, $\t>0$, $\s>1$, and $(x_0,\xi_0)\in U\times\Rd\backslash\{0\}$. Then $(x_0,\xi_0)\not \in {\WF}_{\{\t,\s\}}(u)$ (resp. $(x_0,\xi_0)\not \in {\WF}_{(\t,\s)}(u)$) if there exists a conic neighborhood $\Gamma$ of $\xi_0$, an open neighborhood $\Omega $ of $x_0$, and a bounded sequence $\{u_N\}_{N\in \N}$ in $\E'(U)$ such that $u_N=u$ on $\Omega$ and \eqref{uslov3'} holds for all $\xi \in \Gamma$, $|\xi|\geq B>0$, and for some constants $A,h>0$ (resp. for every $h>0$ there exists $A>0$).
\end{de}

Let $u\in \D'(U)$. Then, immediately follows  that ${\WF}_{\{\t,\s\}}(u)$
is a closed subset of $U\times\Rd\backslash\{0\}$. Note that for $\t>0$ and $\s>1$
$$
{\WF}_{\{\t,\s\}}(u)\subseteq{\WF}_{\{1,1\}}(u)={\WF}_A (u),\;\;\;
u\in \D'(U),
$$
where ${\WF}_A (u)$ denoted the analytic wave front set of a distribution $u\in \D'(U)$, cf. \cite{HermanderKnjiga}.
We refer to \cite{PTT-02} for the relation between the wave front set in Definition \ref{Wf_t_s1}
and other commonly used wave front sets.

The following theorem states that in the definition of
${\WF}_{\t,\s}(u)$ a  bounded sequence  of cut-off functions $\{u_N\}_{N\in \N} \subset \E'(U)$  can be replaced by a single function from $\D_{\t,\s}(U)$.

\begin{te}\cite{PTT-03}
\label{TeoremaKarakterizacija}
Let  $u\in \D'(U)$,  $\t>0$, $\s>1$, and let $(x_0,\xi_0)\in U\times\Rd\backslash\{0\}$.
Then $(x_0,\xi_0)\not \in \WF_{\{\t,\s\}}(u)$ (resp. $(x_0,\xi_0)\not \in {\WF}_{(\t,\s)}(u)$)  if and only if there exists a conic
neighborhood $\Gamma_0$ of $\xi_0$, a compact neighborhood $K$ of $x_0$ and
$\phi\in \D_{\{\t,\s\}}^K$ (resp. $\phi\in \D_{(\t,\s)}^K $) such that $\phi=1$ on a neighborhood of $x_0$, and
\be
\label{SingsuplemaUslovRevisited}
|\widehat{\phi u}(\xi)|\leq A \frac{h^{N^{\s}} N^{\t N^{\s}}}{|\xi|^N},\quad N\in {\N}\,,\xi\in \Gamma_0,
\ee
for some $A,h>0$ (resp. for every $h>0$ there exists $A>0$).
\end{te}

In Theorem \ref{TeoremaKarakterizacija} it is implicitly assumed that the definition is independent of the choice of a cutoff function. However, this is a nontrivial fact, and the proof is based on the properties of the extended associated function.
In the next theorem we actually prove that the characterization of
both Roumieu and Beurling wave front sets, $\WF_{\{\t,\s\}}$ and $\WF_{(\t,\s)}$ respectively,
given in Theorem \ref{TeoremaKarakterizacija} is independent on the choice of the corresponding cutoff functions.

\begin{te}
\label{NezavisnostWFThm}
Let $u\in \D'(U)$, $\t>0$, $\s>1$. The following assertions are equivalent:
\begin{itemize}
\item[i)] $(x_0,\xi_0)\not \in \WF_{\{\t,\s\}}(u)$ (resp. $(x_0,\xi_0)\not \in \WF_{(\t,\s)}(u)$)
\item[ii)] There exists a conic
neighborhood  $\Gamma$ of $\xi_0$, a compact neighborhood
$ K $ of $x_0$ such that for every $\phi\in \D_{\{\t,\s\}}^K$ (resp. $\phi\in \D_{(\t,\s)}^K$) there exists $A,h>0$ (resp. for every $h>0$ there exists $A>0$) such that
\be
\label{TeoremaWFUslov}
|\widehat{\phi u}(\xi)|\leq A e^{-T_{\t,\s,h}(|\xi|)},\xi\in \Gamma.
\ee
\end{itemize}
\end{te}

\begin{proof} Take arbitrary $h>0$. Since the proofs for the Beurling and Roumieu case are similar, we prove theorem only for Roumieu type wave front sets and leave the proof of the Beurling case for the reader.

$ii)=>i)$ Let $\Gamma$ be conic neighborhood of $\xi_0$, $K$ a compact neighborhood of $x_0$, such that for every $\phi\in \D_{(\t,\s)}^K$ and $h>0$ there exits $A_{\phi,h}$ such that \eqref{TeoremaWFUslov} holds. Then, by setting $\phi=\psi$, $\psi\in \D_{(\t,\s)}^K$,  $\psi=1$ on some neighborhood of $x_0$, we have that for every $h>0$ there exists $A_h>0$ so that
$$
|\widehat{\psi u}(\xi)|\leq A_h e^{-T_{\t,\s,h}(|\xi|)} = A_h \inf_{N\in \N}\frac{h^{N^{\s}} N^{\t N^{\s}}}{|\xi|^N},\,\xi\in \Gamma\,,
$$ and the conclusion follows.

$i)=>ii)$ Let $(x_0,\xi_0)\not \in WF_{(\t,\s)}(u)$, i.e., there exists conical neighborhood $\Gamma_1$ of $\xi_0$, compact neighborhood $K_1$ of $x_0$, and $\phi\in \D_{(\t,\s)}^{K_1}$, $\phi=1$ on some neighborhood of $x_0$, such that for every $h>0$ there exists $A>0$ so that \eqref{SingsuplemaUslovRevisited} holds.

Let $K\subset K_1$ be a compact neighborhood of $x_0$ such that $\phi(x)=1$, $x\in K$. Moreover, let $\Gamma$ be a conic neighborhood of $\xi_0$ with the closure contained in $\Gamma_1$, and $\varepsilon>0$ such that $\xi-\eta\in \Gamma_1$ when $\xi\in \Gamma$ and $|\eta|<\varepsilon |\xi|$.

Let $\psi\in \D_{(\t,\s)}^{K}$> Then
$$
\widehat{(\psi u)}(\xi)= \widehat{(\psi \phi u)}(\xi)= \Big(\int_{|\eta|<\varepsilon|\xi|}+\int_{|\eta|\geq\varepsilon|\xi|}\Big)\widehat\psi(\eta) \widehat{\phi u}(\xi-\eta)\, d\eta=I_1+I_2,
$$
for $\xi\in \Gamma$.

Since $|\eta|<\varepsilon |\xi|$ implies $|\xi-\eta|\geq |\xi|-|\eta|> (1-\varepsilon)|\xi|,$
we obtain
\begin{multline} \label{I_1}
|I_1|=\Big| \int_{|\eta|<\varepsilon |\xi|} \widehat \psi(\eta)\widehat{\phi u}(\xi-\eta)\,d\eta \Big|
\leq A \int_{|\eta|<\varepsilon |\xi|} |{\widehat \psi} (\eta)|  \frac{h^{N^{\s}} N^{\t N^{\s}}}{|\xi-\eta|^{N}} d\eta
\\[1ex]
\leq A  \frac{h^{N^{\s}} N^{\t N^{\s}}}{{((1-\varepsilon)|\xi|)}^{N}} \int_{{\bf R}^d}\exp\Big\{-(\t 2^{\s-1 })^{-\frac{1}{\s-1}}\Big(\frac{\s-1}{\s}\Big)^{\frac{\s}{\s-1}}{W^{-\frac{1}{\s-1}}\Big({\mathfrak R}\Big(\frac{1}{h\sqrt{d}},\langle \eta\rangle\Big)\Big)}{{\ln}}^{\frac{\s}{\s-1}}\langle \eta \rangle \Big\}\\
\leq A \frac{h_1^{N^{\s}} N^{\t N^{\s}}}{|\xi|^{N}}\ ,\quad \xi\in \Gamma, N\in \N,
\end{multline}
for $A>0$ and $\dss h_1=(1-\varepsilon)^{-1} h$, where we have used Theorem \ref{TeoremaPaley} (for $\eta \in \Rd$) and the fact that
$$
 \int_{{\bf R}^d}\exp\Big\{-(\t 2^{\s-1 })^{-\frac{1}{\s-1}}\Big(\frac{\s-1}{\s}\Big)^{\frac{\s}{\s-1}}{W^{-\frac{1}{\s-1}}\Big({\mathfrak R}\Big(\frac{1}{h\sqrt{d}},\langle \eta\rangle\Big)\Big)}{{\ln}}^{\frac{\s}{\s-1}}\langle \eta \rangle \Big\}\leq C_h,
$$
for some $C_h>0$, which follows from \eqref{ocenaIntegral}.

In particular,
$$
|I_1|\leq A\inf_{N\in \N} \frac{h_1^{N^{\s}} N^{\t N^{\s}}}{|\xi|^{N}}= e^{-T_{\t,\s,h_1}(|\xi|)},\quad \xi\in \Gamma.
$$
To estimate $I_2$, we use that  $|\eta|\geq \varepsilon |\xi|$ implies
$$
|\xi-\eta|\leq |\xi|+|\eta|\leq (1+1/\varepsilon)|\eta|.
$$
Since $\phi u\in \E'(U)$, there exists $C,M>0$ such that $\dss |\widehat{\phi u} (\xi)|\leq C \langle \xi \rangle^M$, $\xi \in \Rd$. For a given $N\in \N$, we put $|\alpha|=N+M+d+1$.
Then, applying inequality \eqref{NejednakostZaWF} to $\widehat\psi (\eta)$, $\eta \in \Rd$, we obtain
\begin{multline*}
|I_2|=\Big| \int_{|\eta|\geq \varepsilon |\xi|} \widehat \psi(\eta)\widehat{\phi u}(\xi-\eta)\,d\eta \Big|\\
\leq A\frac{ h^{{(N+M+d+1)}^\s} {(N+M+d+1)^{\t {(N+M+d+1)}^{\s}}}}{(\varepsilon |\xi|)^{N}}
\cdot \int_{|\eta|\geq \varepsilon |\xi|} \langle\eta\rangle^{-M-d-1}  \langle\xi-\eta\rangle^{M} \,d\eta\\
\leq A \frac{ h_2^{N^\s} {N^{\t N^{\s}}}}{|\xi|^{N}} \quad \xi\in \Gamma, N\in \N,
\end{multline*}
for $A,C>0$ and $h_2=C\max\{h, h^{2^{\s-1}}\}$, where in the last inequality we have used $\overline{(M.2)'}$ property of  $M_p^{\t,\s}=p^{\t p^{\s}}$ and
$$
\dss |\alpha|^{\s}+|\beta|^{\s}\leq |\alpha+\beta|^{\s}\leq 2^{\s-1}(|\alpha|^{\s}+|\beta|^{\s}),\quad \alpha,\beta \in \N^d.
$$
Hence we conclude that
$$ |I_2|\leq A e^{-T_{\t,\s,h_2}(|\xi|)},\quad \xi\in \Gamma. $$
This, together with \eqref{I_1} gives \eqref{TeoremaWFUslov} and the proof is finished.
\end{proof}

\section{Proof of Theorem \ref{propozicija}} \label{Sec:Proof}

Let $h,\t>0$ and $\s>1$ be fixed. It is enough to prove that there exist constants such that $H_{\t,\s,h},C_{\t,\s,h}>0$ such that
\begin{multline}
\label{nejednakostzaTeoremu2}
( 2^{\s-1 }\t)^{-\frac{1}{\s-1}}{\Big(\frac{\s-1}{ \s}\Big)^{\frac{\s}{\s-1}}\, {W^{-\frac{1}{\s-1}}({{\mathfrak R}(h,k)})}\,{\ln}^{\frac{\s}{\s-1}}k }-H_{\t,\s,h}\\
\leq \sup_{p\in \N}\ln\frac{h^{p^{\s}}k^{p}}{p^{\t p^{\s}}}\leq {\Big(\frac{\s-1}{\t \s}\Big)^{\frac{1}{\s-1}}\, {W^{-\frac{1}{\s-1}}({{\mathfrak R}(h,k)})}\,{\ln}^{\frac{\s}{\s-1}}k },
\end{multline}
when $k\geq C_{\t,\s,h}>e$. Then \eqref{nejednakostzaTeoremu1} follows after taking exponentials.

To prove the right-hand side of \eqref{nejednakostzaTeoremu2}, note that
\be
\label{pocetnaNejednakost}
\sup_{p\in \N}\ln\frac{h^{p^{\s}}k^{p}}{p^{\t p^{\s}}}\leq \sup_{r>0}\ln\frac{h^{r^{\s}}k^{r}}{r^{\t r^{\s}}},\quad h>0, k>e.
\ee
We compute the supremum on the right-hand side as follows:

Put
\be
\label{maloF}
f(r):=r^{\s}\ln h +r \ln k-\t r^{\s}\ln r,\quad r>0,
\ee and note that the equation $f'(r)=0$, that is,

\be
\label{jednacina0}
\s r^{\s-1}\ln h +\ln k -\t \s r^{\s-1}\ln r-\t r^{\s-1}=0,
\ee
can be rewritten in the form
\be
\label{jednacina1}
 r^{\s-1}\ln \frac{e^{\frac{\s-1}{\s}}r^{\s-1}}{h^{\frac{\s-1}{\t}}}=\frac{\s-1}{\t \s}\ln k,\nonumber
\ee or equivalently
\be
\label{jednacina2}
 \frac{e^{\frac{\s-1}{\s}}r^{\s-1}}{h^{\frac{\s-1}{\t}}}\ln \frac{e^{\frac{\s-1}{\s}}r^{\s-1}}{h^{\frac{\s-1}{\t}}}={\mathfrak R}(h,k)\quad r,h>0,\,k>e.
\ee
Applying the Lambert function to \eqref{jednacina2} and using the fact that
$$
W(a \ln a)=W (e^{\ln a}\ln a )=\ln a,\quad a>e,
$$
we obtain the equation
$$\ln \frac{e^{\frac{\s-1}{\s}}r^{\s-1}}{h^{\frac{\s-1}{\t}}}=W({\mathfrak R}(h,k)),\quad h>0, k>e,$$
and the solution $r_0$ of the equation \eqref{jednacina0} is given by
\be
\label{R0}
r_0:=h^{1/\t}e^{\frac{1}{\s-1}W({\mathfrak R}(h,k))-1/\s}.
\ee
To shorten the notation, in the sequel we write $\mathfrak R$ instead of $ {\mathfrak R}(h,k)$.

Now we have
\be
\label{racun1}
\sup_{r>0}\ln\frac{h^{r^{\s}}k^{r}}{M_r^{\t,\s}}=f(r_0)=r_0^{\s}(\ln h-\t \ln r_0)+r_0\ln k,\quad k>e.
\ee
For the first summand in \eqref{racun1} we obtain
\begin{multline*}
r_0^{\s}(\ln h-\t \ln r_0)=\t h^{\frac{\s}{\t}}e^{\frac{\s}{\s-1}W({\mathfrak R})-1}\ln\Big( e^{-\frac{1}{\s-1}W({\mathfrak R})+\frac{1}{\s}}\Big) \\
=\frac{\t}{\s} h^{\frac{\s}{\t}}e^{\frac{\s}{\s-1}W({\mathfrak R})-1}-\frac{\t}{\s-1}h^{\frac{\s}{\t}}e^{\frac{1}{\s-1}W({\mathfrak R})-1}\Big(e^{W({\mathfrak R})} W({\mathfrak R})\Big)\\
=\frac{\t}{\s} h^{\frac{\s}{\t}}e^{\frac{\s}{\s-1}W({\mathfrak R})-1} -\frac{1}{\s}h^{\frac{1}{\t}}e^{\frac{1}{\s-1}W({\mathfrak R})-\frac{1}{\s}} \ln k,\quad k>e,
\end{multline*}
where for the last equality we use
\be
\label{LambertR}
W({\mathfrak R}(h,k))e ^{W({\mathfrak R}(h,k))}={\mathfrak R}(h,k)\quad h>0, k>e,
\ee
see \eqref{osobinaLambert}.
Therefore we obtain
\be
\label{racun2}
\sup_{r>0}\ln\frac{h^{r^{\s}}k^{r}}{M_r^{\t,\s}}=\frac{\t}{\s} h^{\frac{\s}{\t}}e^{\frac{\s}{\s-1}W({\mathfrak R})-1} +\frac{\s-1}{\s}h^{\frac{1}{\t}}e^{\frac{1}{\s-1}W({\mathfrak R})-\frac{1}{\s}} \ln k.
\ee
Since
$$
\ln k = h^{\frac{\s-1}{\t}}e^{-\frac{\s-1}{\s}}\frac{\t \s}{\s-1} {\mathfrak R}(h,k),\quad h>0,\, k>e,
$$
using $\dss e^{W({\mathfrak R})}=\frac{{\mathfrak R}}{W({\mathfrak R})}$ again, we may rewrite \eqref{racun2} as
\begin{multline}
\label{racunGornjaOcena}
\sup_{r>0}\ln\frac{h^{r^{\s}}k^{r}}{M_r^{\t,\s}}= \frac{\t}{e} h^{\frac{\s}{\t}} \Big(\frac{1}{\s}\Big(\frac{{\mathfrak R}}{W({\mathfrak R})}\Big)^{\frac{\s}{\s-1}}+\Big(\frac{{\mathfrak R}}{W({\mathfrak R})}\Big)^{\frac{1}{\s-1}}{\mathfrak R}\Big)\\
=\frac{\t}{e} h^{\frac{\s}{\t}}{\mathfrak R}^{\frac{\s}{\s-1}}\Big(\frac{1}{W({\mathfrak R})}\Big)^{\frac{1}{\s-1}}\Big(\frac{1}{\s}\frac{1}{W({\mathfrak R})}+1\Big)\\
=\t^{-\frac{1}{\s-1}} \Big(\frac{\s-1}{\s}\Big)^{\frac{\s}{\s-1}}{W^{-\frac{\s}{\s-1}}({\mathfrak R}(h,k))}\ln^{\frac{\s}{\s-1}} k\Big(\frac{1}{\s}+ W({\mathfrak R}(h,k))\Big),
\end{multline}
for $h>0$ and $k>e$.

Next, we put $\widetilde{C_{\t,\s,h}}:= e^{\t h^{\frac{\s-1}{\t}}}$, $h>0$, and prove that
\be
\label{poslednjaNejednakost}
\frac{1}{\s}+ W({\mathfrak R})\leq \frac{\s}{\s-1} W({\mathfrak R}),
\ee
if and only if $k\geq\widetilde{C_{\t,\s,h}}$.

In fact, \eqref{poslednjaNejednakost} is equivalent to
\be
\label{poslednjaNejednakost1}
W({\mathfrak R})\geq \frac{\s-1}{\s}.
\ee
Applying the inverse of the Lambert $W$ function we obtain
$${\mathfrak R}=h^{-\frac{\s-1}{\t}}e^{\frac{\s-1}{\s}}\frac{\s-1}{\t \s}\ln k\geq \frac{\s-1}{\s} e ^{\frac{\s-1}{\s}},\quad h>0,$$
which is true if and only if $k\geq \widetilde{C_{\t,\s,h}}$.

Hence, by \eqref{pocetnaNejednakost}, \eqref{racunGornjaOcena} and \eqref{poslednjaNejednakost}, we conclude that
$$
\sup_{r>0}\ln\frac{h^{r^{\s}}k^{r}}{M_r^{\t,\s}} \leq \Big(\frac{\s-1}{\t \s}\Big)^{\frac{1}{\s-1}} {W^{-\frac{1}{\s-1}}({\mathfrak R})}\,{\ln}^{\frac{\s}{\s-1}} k,\quad k\geq \widetilde{C_{\t,\s,h}},
$$
which gives the right-hand side of \eqref{nejednakostzaTeoremu2} for $k> \max\{e, \widetilde{C_{\t,\s,h}}\}$.

\par

To prove the left-hand side of \eqref{nejednakostzaTeoremu2}, recall that the first derivative of the function $f$
introduced in \eqref{maloF} is given by
$$ f'(r)= - \t \s r^{\s-1}\ln\frac{r e^{1/\s}}{h^{1/\t}}+\ln k\quad r>0, h>0, k>e.$$
Thus the $n^{th}$ derivative of $f$ is given by
$$
f^{(n)}(r)=- \t \s(\s-1)\dots (\s-n+1) r^{\s-n}\ln\Big(\frac{r}{h^{1/\t}} \exp\Big\{\sum_{l=0}^{n-1}\frac{1}{\s-l}\Big\}\Big),
$$
for $2\leq n\leq \lfloor\s\rfloor$, $h>0$ and $r>0$.

Moreover, if $\s>1$ we have
\begin{multline}
\label{NtiIzvod1}
f^{(\lfloor \s \rfloor+1)}(r)=- \t \s(\s-1)\dots (\s-\lfloor\s\rfloor) r^{\s-\lfloor \s \rfloor-1}\ln\Big(\frac{r}{h^{1/\t}} \exp\Big\{\sum_{l=0}^{\lfloor \s\rfloor}\frac{1}{\s-l}\Big\}\Big),\\ h>0,\, r>0,
\end{multline}
when $\s\not\in \N$, and, for $\s\in \N$,
\be
\label{NtiIzvod2}
f^{(\s+1)}(r)=-\t \s! \frac{1}{r},\quad r>0.
\ee

\par

Set $r_1:=\lceil r_0 \rceil $, where $r_0$ is given by \eqref{R0}. Then Taylor's formula implies
\begin{multline}
\label{Taylor}
f(r_1)=f(r_0)+\sum_{n=1}^{\lfloor\s\rfloor}\frac{f^{(n)}(r_0)}{n!}(r_1-r_0)^n+\frac{f^{(\lfloor \s \rfloor+1)}(r_2)}{(\lfloor \s \rfloor +1)!}(r_1-r_0)^{\lfloor \s \rfloor+1}
\end{multline} for some $r_2\in (r_0,r_1)$.

Now we estimate the third summand in \eqref{Taylor}. We first consider the case $\s>1$ and $\s\not \in \N$.
Let $\varepsilon =\lfloor \s \rfloor-\s+1$ and $0<\varepsilon'<\varepsilon$.
Then, by \eqref{NtiIzvod1}, we have
\begin{multline*}
\frac{f^{(\lfloor \s \rfloor+1)}(r_2)}{(\lfloor \s \rfloor+1)!}(r_1-r_0)^{\lfloor \s \rfloor+1}\\
= -\t\frac{\s(\s-1)\dots(1-\varepsilon)}{(\lfloor \s \rfloor+1)!}r_2^{-\varepsilon}\ln\Big(\frac{r_2}{h^{1/\t}} \exp\Big\{\sum_{l=0}^{\lfloor \s \rfloor}\frac{1}{\s-l}\Big\}\Big)(r_1-r_0)^{\lfloor \s \rfloor+1},
\end{multline*}
for $ h>0$ and $r_2\in (r_0,r_1)$. Since
$$
\ln\Big(\frac{r_2}{h^{1/\t}} \exp\Big\{\sum_{l=0}^{\lfloor \s \rfloor}\frac{1}{\s-l}\Big\}\Big)\leq  L_{\t,\s,h} r_2^{\varepsilon'},\quad \,r_2>0,
$$
for some  constant $L_{\t,\s,h}>0$. Using that $r_1-r_0<1$, $r_0<r_2$ and $\varepsilon'<\varepsilon$, we obtain
\begin{multline}
\label{ocenaTreciSabirak0}
\t\frac{\s(\s-1)\dots(1-\varepsilon)}{(\lfloor \s \rfloor+1)!}r_2^{-\varepsilon}\ln\Big(\frac{r_2}{h^{1/\t}} \exp\Big\{\sum_{l=0}^{\lfloor \s \rfloor}\frac{1}{\s-l}\Big\}\Big)(r_1-r_0)^{\lfloor \s \rfloor +1}\\ \leq\t L_{\t,\s,h} r_2^{\varepsilon'-\varepsilon}
\leq\t L_{\t,\s,h} r_0^{\varepsilon'-\varepsilon}\leq \t L_{\t,\s,h} h^{(\varepsilon'-\varepsilon)/\t} e^{-(\varepsilon'-\varepsilon)/\s} e^{(\varepsilon'-\varepsilon)W(\mathfrak R)(h,k)}\\
\leq\t L_{\t,\s,h} h^{(\varepsilon'-\varepsilon)/\t} e^{-(\varepsilon'-\varepsilon)/\s},\quad h>0,\, k>e,
\end{multline}
where the last inequality follows from the fact that $W(\mathfrak R)(h,k)>0$ when $h>0$ and $k>e$.

Now we estimate the third summand in \eqref{Taylor} when $\s>1$ and $\s\in \N$. We use \eqref{NtiIzvod2}
to obtain
\be \label{ocenaTreciSabirak}
\frac{\t}{\s+1}\frac{1}{r_2}(r_1-r_0)^{\s+1}
<\frac{\t}{\s+1}\frac{1}{r_0}=h^{-1/\t} e^{1/\s} e^{-\frac{1}{\s-1}W({\mathfrak R})}<\frac{\t}{\s+1}h^{-1/\t} e^{1/\s},
\ee
since $r_2>r_0$ and $r_1-r_0<1$.

Therefore, the third summand in \eqref{Taylor} is estimated by \eqref{ocenaTreciSabirak0} and
\eqref{ocenaTreciSabirak}.

Next, we estimate the second term in \eqref{Taylor}.
Since $f'(r_0)=0$  that  term is equal to zero when $1<\s<2$, so it remains to consider the case $\s\geq 2$.

We have
\begin{multline}
\label{racunTaylor}
\t \sum_{n=2}^{\lfloor\s\rfloor}\frac{\s(\s-1)\dots (\s-n+1)}{n!} r_0^{\s-n}\ln\frac{r_0 \exp\{{\sum_{l=0}^{n-1}\frac{1}{\s-l}}\}}{h^{1/\t}}(r_1-r_0)^n \\
\leq \t \sum_{n=2}^{\s}\frac{\s(\s-1)\dots (\s-n+1)}{n!} h^{\frac{\s-n}{\t}} e^{\frac{\s-n}{\s-1}W(\mathfrak R)-\frac{\s-n}{\s}}\Big( \frac{1}{\s-1}W({\mathfrak R})+ {\sum_{l=1}^{n-1}\frac{1}{\s-l}}\Big)\\
\leq \widetilde{L_{\t,\s,h}}\t e^{\frac{\s-2}{\s-1}W(\mathfrak R)}\Big( \frac{1}{\s-1}W({\mathfrak R})+ {\sum_{l=1}^{\s-1}\frac{1}{\s-l}}\Big),
\end{multline}
where $\dss \widetilde{L_{\t,\s,h}}=\sum_{n=2}^{\s}\frac{\s(\s-1)\dots (\s-n+1)}{n!} h^{\frac{\s-n}{\t}} e^{-\frac{\s-n}{\s}}$.

By the similar procedure as in the in the proof of \eqref{poslednjaNejednakost} we can find  $B_{\t,\s,h}>0$ such that
\be
\label{racunTaylor1}
\frac{1}{\s-1}W({\mathfrak R})+ {\sum_{l=1}^{\s-1}\frac{1}{\s-l}}\leq \frac{\s}{\s-1} W({\mathfrak R}).
\ee
if and only if $k\geq B_{\t,\s,h}$.

We use the simple inequality
$$
e^{(c-\varepsilon) x}<\frac{1}{D}\cdot e^{c x},\,\,\, \text{if and only if}\,\,\,  x> \frac{\ln D}{\varepsilon},\,\,\, c,D,\varepsilon>0,
$$
with $\dss c=\s/(\s-1)$, $ \dss \varepsilon=2/(\s-1), $ $\dss D=2h^{-\s/\t} e^{\s/(\s-1)} \widetilde{L_{\t,\s,h}},$
where $\widetilde{L_{\t,\s,h}}$ is given as above (cf. \eqref{racunTaylor}). We obtain
\be
\label{racunTaylor2}
e^{\frac{\s-2}{\s-1}W({\mathfrak R})}\leq \frac{1}{2 \widetilde{L_{\t,\s,h}}}\t^{-\frac{\s}{\s-1}}\Big(\frac{\s-1}{\s}\Big)^{\frac{\s}{\s-1}+1}W^{-{\frac{\s}{\s-1}}}({\mathfrak R})\ln^{\frac{\s}{\s-1}} k,
\ee
if and only if $\dss W({\mathfrak R})>\frac{\ln D}{\varepsilon}$. Similarly as in \eqref{poslednjaNejednakost1}, inequality \eqref{racunTaylor2} holds if and only if $\dss k>\widetilde{B_{\t,\s,h}}$ for some $\widetilde{B_{\t,\s,h}}>0$. Note that for \eqref{racunTaylor2} we also use \eqref{LambertR}.

Therefore, by \eqref{racunTaylor}, \eqref{racunTaylor1} and \eqref{racunTaylor2} it follows
that the second term in \eqref{Taylor} is bounded by
$$ (2^{\s-1}\t)^{-{\frac{1}{\s-1}}}\Big(\frac{\s-1}{\s}\Big)^{\frac{\s}{\s-1}}W^{-{\frac{1}{\s-1}}}({\mathfrak R})\ln^{\frac{\s}{\s-1}} k.$$
Finally, by \eqref{racunGornjaOcena}, \eqref{Taylor}, \eqref{ocenaTreciSabirak0} and \eqref{ocenaTreciSabirak}  we have
\begin{multline*}
\sup_{p\in \N}\ln \frac{h^{p^{\s}}k^p}{p^{\t p^{\s}}}\geq f(r_1)\geq \t^{-\frac{1}{\s-1}} \Big(\frac{\s-1}{\s}\Big)^{\frac{\s}{\s-1}}{W^{-\frac{\s}{\s-1}}({\mathfrak R})}\ln^{\frac{\s}{\s-1}} k\Big(\frac{1}{\s}+ W({\mathfrak R})\Big)\\
 -  (2^{\s-1}\t)^{-{\frac{1}{\s-1}}}\Big(\frac{\s-1}{\s}\Big)^{\frac{\s}{\s-1}}W^{-{\frac{1}{\s-1}}}({\mathfrak R})\ln^{\frac{\s}{\s-1}} k-H_{\t,\s,h}\\
\geq  (2^{\s-1}\t)^{-{\frac{1}{\s-1}}}\Big(\frac{\s-1}{\s}\Big)^{\frac{\s}{\s-1}}W^{-{\frac{1}{\s-1}}}({\mathfrak R})\ln^{\frac{\s}{\s-1}} k-H_{\t,\s,h},
\end{multline*} for suitable $H_{\t,\s,h}>0$ and $k>\max\{e,B_{\t,\s,h},\widetilde{B_{\t,\s,h}}\}$.
Therefore, we obtain the left-hand side of \eqref{nejednakostzaTeoremu2} and the proof is finished.

\subsection*{Acknowledgment}
This research is supported by Ministry of Education, Science and
Technological Development of Serbia through the Project no. 174024.
\par

\end{document}